\documentclass[12pt,reqno]{amsart}

\usepackage{amssymb}

\newtheorem{theorem}{Theorem}[section]
\newtheorem{proposition}[theorem]{Proposition}
\newtheorem{lemma}[theorem]{Lemma}

\theoremstyle{remark}

\newtheorem{remark}[theorem]{Remark}

\numberwithin{equation}{section}

\begin{document}

\title[Automorphism groups, Levi reduction and invariant connections]{Automorphism group of 
principal bundles, Levi reduction and invariant connections}

\author[Indranil Biswas]{Indranil Biswas}

\address{School of Mathematics, Tata Institute of Fundamental Research,
Homi Bhabha Road, Mumbai 400005, India and
Mathematics Department, EISTI-University Paris-Seine, Avenue du parc, 95000,
Cergy-Pontoise, France}

\email{indranil@math.tifr.res.in}

\author[Francois-Xavier Machu]{Francois-Xavier Machu}

\address{Mathematics Department, EISTI-University Paris-Seine, Avenue du parc, 95000, 
Cergy-Pontoise, France}

\email{fmu@eisti.eu}



\date{}

\begin{abstract}
Let $M$ be a compact connected complex manifold and $G$ a connected reductive complex affine
algebraic group. Let $E_G$ be a holomorphic principal $G$--bundle over $M$ and $T\, \subset\, G$ a
torus containing the connected component of the center of $G$. Let $N$ (respectively, $C$) be
the normalizer (respectively, centralizer) of $T$ in $G$, and let $W$
be the Weyl group $N/C$ for $T$. We prove that there is a
natural bijective correspondence between the following two:
\begin{enumerate}
\item Torus subbundles $\mathbb T$ of ${\rm Ad}(E_G)$ such that for some
(hence every) $x\, \in\, M$, the 
fiber ${\mathbb T}_x$ lies in the conjugacy class of tori in ${\rm Ad}(E_G)$ determined by $T$.

\item Quadruples of the form $(E_W,\, \phi,\, E'_C,\, \tau)$, where $\phi\, :\, E_W\,
\longrightarrow\, M$ is a principal $W$--bundle, $\phi^*E_G\, \supset\, E'_C\,
\stackrel{\psi}{\longrightarrow}\, E_W$ is a holomorphic
reduction of structure group of $\phi^* E_G$ to $C$, and
$$
\tau\,:\, E'_C\times N \, \longrightarrow\, E'_C
$$
is a holomorphic action of $N$ on $E'_C$ extending the natural action of $C$ on $E'_C$,
such that the composition $\psi\circ\tau$ coincides with the composition of the
quotient map $E'_C\times N\, \longrightarrow\, (E'_C/C)\times (N/C)\,=\,
(E'_C\times N)/(C\times C)$ with the natural map $(E'_C/C)\times (N/C)\, \longrightarrow\, E_W$.
\end{enumerate}
The composition of maps $E'_C\, \stackrel{\psi}{\longrightarrow}\, E_W \,
\stackrel{\phi}{\longrightarrow}\,
M$ defines a principal $N$--bundle on $M$. This principal $N$--bundle $E_N$ is
a reduction of structure group of $E_G$ to $N$. Given a complex connection $\nabla$ on $E_G$, we give a
necessary and sufficient condition for $\nabla$ to be induced by a connection on $E_N$. This
criterion relates Hermitian--Einstein connections on $E_G$ and $E'_C$ in a very precise manner.
\end{abstract}

\maketitle

\tableofcontents

\section{Introduction}\label{sec0}

Let $M$ be a compact connected complex manifold. Take a holomorphic vector bundle $E$ on $M$. In 
\cite{BCW} the following question was addressed: When is
the vector bundle $E$ the direct image of a vector bundle 
over an \'etale cover of $M$? The main result of \cite{BCW} described all possible way $E$ is 
realized as the direct image of a vector bundle over an \'etale cover of $M$.
The main result of \cite{BCW} says that they are 
parametrized by the subbundles of the adjoint bundle $\text{Ad}(E)\, \longrightarrow\, M$ whose 
fibers are tori. To explain this with more details, given any triple $(Y,\, \beta,\, F)$, where 
$\beta\, :\, Y\, \longrightarrow\, M$ is an \'etale covering ($Y$ need not be connected) and $F$ 
is a holomorphic vector bundle on $Y$ such that $E\,=\, \beta_*F$, we construct a torus sub-bundle
of $\text{Ad}(E)$; this sub-bundle is in fact the invertible part of $\beta_*{\mathcal O}_Y\, 
\subset\, \text{End}(\beta_*F)$. Conversely, given a sub-bundle of $\text{Ad}(E)$ with the 
typical fiber being a torus, we construct a triple $(Y,\, \beta,\, F)$ of the above form such 
that $E\,=\, \beta_*F$. In \cite{BM}, these results were generalized to the context of parabolic 
(orbifold) vector bundles over any Riemann surface; see \cite{DP} for a somewhat related question.

Our aim here is to formulate and address the question in the context of principal bundles.
Since direct image of a principal bundle does not quite make sense, a reformulation is warranted.

Let $G$ be a connected reductive complex affine algebraic group.
Fix a complex torus $T\, \subset\, G$ that contains the connected 
component, containing the identity element, of the center of $G$. Denote the normalizer 
(respectively, centralizer) of $T$ in $G$ by $N$ (respectively, $C$). This $C$ is a Levi factor
of a parabolic subgroup of $G$. The quotient $N/C$ is a finite group, which we shall
denote by $W$.

We prove the following (see Theorem \ref{thmse1}):

\begin{theorem}\label{thmi0}
Take a holomorphic principal $G$--bundle $E_G$ over $M$. There
is a natural bijective correspondence between the following two:
\begin{enumerate}
\item Torus subbundles $\mathbb T$ of ${\rm Ad}(E_G)$ such that for some
(hence every) $x\, \in\, M$, the 
fiber ${\mathbb T}_x$ lies in the conjugacy class of tori in ${\rm Ad}(E_G)$ determined by $T$.

\item Quadruples of the form $(E_W,\, \phi,\, E'_C,\, \tau)$, where $\phi\, :\, E_W\,
\longrightarrow\, M$ is a principal $W$--bundle, $\phi^*E_G\, \supset\, E'_C\,
\stackrel{\psi}{\longrightarrow}\, E_W$ is a holomorphic
reduction of structure group of the principal $G$--bundle $\phi^* E_G$ to the subgroup $C$, and
$$
\tau\,:\, E'_C\times N \, \longrightarrow\, E'_C
$$
is a holomorphic action of $N$ on $E'_C$ extending the natural action of $C$ on the principal
$C$--bundle $E'_C$, such that the diagram of maps
$$
\begin{matrix}
E'_C\times N & \stackrel{\tau}{\longrightarrow} & E'_C\\
\Big\downarrow && ~\Big\downarrow\psi \\
E_W\times W := (E'_C/C)\times (N/C)& \longrightarrow & E_W
\end{matrix}
$$
is commutative, where $E'_C\times N\, \longrightarrow\, (E'_C/C)\times (N/C)\,=\,
(E'_C\times N)/(C\times C)$ is the quotient map, and $\psi$ is the natural
projection.
\end{enumerate}
\end{theorem}

If we set $G\,=\, \text{GL}(r, {\mathbb C})$ in Theorem \ref{thmi0}, then the
above mentioned result of \cite{BCW} is obtained.

Consider the composition of maps $E'_C\, \stackrel{\psi}{\longrightarrow}\, E_W \,
\stackrel{\phi}{\longrightarrow}\, M$ in (2). The action of $N$ on $E'_C$ and this
composition of maps together produce a holomorphic principal $N$--bundle over $M$. This
principal $N$--bundle over $M$ will be denoted by $E_N$. Since $\phi^*E_G\, =\, E_W\times_M E_G$,
we have a natural projection $\phi^*E_G\, \longrightarrow\, E_G$. Now using the composition of maps
\begin{equation}\label{inte}
E_N\, =\, E'_C\, \hookrightarrow\, \phi^*E_G \, \longrightarrow\, E_G\, ,
\end{equation}
the principal $N$--bundle $E_N$ is a holomorphic reduction of structure group of the
principal $G$--bundle $E_G$ to the subgroup $N\, \subset\, G$. Therefore, a complex connection
on $E_N$ induces a complex connection on $E_G$.

Given a complex connection on $E_G$, it is natural to ask whether it is induced by a
complex connection on $E_N$. 

We prove the following criterion for it (see Theorem \ref{thm1} and Remark \ref{rem-c}):

\begin{theorem}\label{thmi1}
Given a complex connection $D_0$ on $E_G$, let $D_1$ be the complex connection on the associated 
adjoint bundle ${\rm Ad}(E_G)$ induced by $D_0$. The connection $D_0$ on $E_G$ is induced by a 
complex connection on the principal $N$--bundle $E_N$ if and only if the induced connection 
$D_1$ on ${\rm Ad}(E_G)$ preserves the corresponding torus subbundle ${\mathbb T}\, \subset\, 
{\rm Ad}(E_G)$ in (1) of Theorem \ref{thmi0}.

When the connection $D_1$ on ${\rm Ad}(E_G)$ induced by $D_0$ preserves the torus subbundle
${\mathbb T}\, \subset\, {\rm Ad}(E_G)$, the connection $D_0$ is holomorphic if and only
if the connection on $E_N$ inducing $D_0$ is holomorphic.
\end{theorem}

A complex connection on $E_N$ defines a complex connection on the principal $C$--bundle
$E'_C$ on $E_W$, because $E_N\, =\, E'_C$ (see \eqref{inte}) and the map $\phi$ is \'etale.

Now assume that $M$ is K\"ahler; fix a K\"ahler form $\omega$ on $M$ in order to define degree
of a torsionfree coherent analytic sheaf on $M$. This enables us to define stable and
polystable principal $G$--bundles on $M$. Fix a maximal compact subgroup $K_G\, \subset\, G$
to define the Hermitian--Einstein equation for principal $G$--bundles. So $K_C\, :=\, C\bigcap K_G$
is a maximal compact subgroup of $C$. A holomorphic principal $G$--bundle on $M$
admits a Hermitian--Einstein connection if and only if it is polystable \cite{UY},
\cite{Do1}, \cite{RS}, \cite{AB}. The pulled back form $\phi^*\omega$ on $E_W$ is
K\"ahler. However, $E_W$ need not be connected. Polystable bundles and Hermitian--Einstein
connections on bundles over $E_W$ are defined in a suitable way.

We prove the following (see Proposition \ref{prop3}):

\begin{proposition}\label{prop3i}
Take $E_G$ and $E'_C$ as in Theorem \ref{thmi0}.
Assume that the principal $G$--bundle $E_G$ on $M$ is polystable. Let $\nabla$
be the Hermitian--Einstein connection on $E_G$. Then the following two
hold:
\begin{enumerate}

\item The principal $C$--bundle $E'_C$ on $E_W$ is polystable.

\item The Hermitian--Einstein connection $\phi^*\nabla$ on $\phi^*E_G$ preserves
the reduction $E'_C$ of structure group of the principal $G$--bundle $\phi^*E_G$ to
the subgroup $C\, \subset\, G$. Furthermore, the connection on $E'_C$
given by $\phi^*\nabla$ is Hermitian--Einstein.
\end{enumerate}
\end{proposition}

Let $z({\mathfrak k})$ denote the center of the Lie algebra of $K_G$. Let $z({\mathfrak k}_c)$ 
be the center of the Lie algebra of the maximal compact subgroup $K_C \,=\, K_G\bigcap C$ of 
$C$. We have $z({\mathfrak k})\, \subset\, z({\mathfrak k}_c)$.

We also prove the following (see Proposition \ref{prop4}):

\begin{proposition}\label{prop4i}
Take $E_G$ and $E'_C$ as in Theorem \ref{thmi0}.
Assume that the principal $C$--bundle $E'_C$ over $E_W$ is polystable. Let $\nabla$
be the Hermitian--Einstein connection on $E'_C$. Assume that the element of $z({\mathfrak k}_c)$
given by the curvature of $\nabla$ lies in
the subspace $z({\mathfrak k})$. Then the following two hold:
\begin{enumerate}

\item The principal $G$--bundle $E_G$ on $M$ is polystable.

\item The Hermitian--Einstein connection on $E_G$ is given by $\nabla$.
\end{enumerate}
\end{proposition}

As indicated in Section \ref{seeq}, all the above results extend to the equivariant
set-up. This means that the results of \cite{BM} extend to the context of complex
reductive affine algebraic groups.

\section{Torus subbundle and Levi reduction of structure group of a principal bundle}

\subsection{A Levi reduction from a torus subbundle}\label{se2.1}

Let $G$ be a connected complex reductive affine algebraic group. A parabolic subgroup of $G$ is 
a Zariski closed connected subgroup $P\, \subset\, G$ such that the quotient variety $G/P$ is 
projective. The unipotent radical of a parabolic subgroup $P\, \subset\, G$ is denoted by 
$R_u(P)$. The quotient $P/R_u(P)$ is a reductive affine complex algebraic group. A connected reductive 
complex algebraic subgroup $L(P)\, \subset\, P$ is called a Levi factor of $P$ if the composition
of maps
$$
L(P)\, \hookrightarrow\, P\, \longrightarrow\, P/R_u(P)
$$
is an isomorphism \cite[p. 158, \S~11.22]{Bo}. There are Levi factors of $P$; any two Levi factors 
of $P$ differ by the inner automorphism of $P$ produced by an element of $R_u(P)$ \cite[p. 158, \S~11.23]{Bo}, 
\cite[\S~30.2, p. 184]{Hu}.

Fix a Borel subgroup $B_G\, \subset\, G$, and also fix a maximal torus $T_G\, \subset\, B_G$. 
Given any parabolic subgroup $P\,\subset\, G$, there is some element $g_0\, \in\, G$ such that 
we have $B_G\,\subset\, g^{-1}_0 Pg_0$ \cite[p.~134, Theorem 21.3]{Hu}. Henceforth, whenever we 
consider a parabolic subgroup of $G$, we would assume that $P\, \supset\, B_G$. The connected 
component of the center of $G$ containing the identity element will be denoted by $Z_0(G)$; this 
$Z_0(G)$ is isomorphic to a product of copies of the multiplicative group ${\mathbb C}^*\, :=\, 
{\mathbb C}\setminus\{0\}$.

Let $M$ be a compact connected complex manifold. Let
\begin{equation}\label{p0}
p_0\, :\, E_G\, \longrightarrow\, M
\end{equation}
be a holomorphic principal $G$--bundle over $M$; this means that $E_G$ is a holomorphic fiber bundle over $M$
equipped with a holomorphic right-action of $G$
$$
q_0\, :\, E_G\times G\,\longrightarrow\, E_G
$$
such that the above projection $p_0$ is $G$--invariant and, furthermore, the resulting map to the fiber product
$$
E_G\times G \, \longrightarrow\, E_G\times_M E_G\, , \ \ (y,\, z) \, \longrightarrow\, (y,\,
q_0(y,\, z))
$$
is a biholomorphism. For notational convenience, for any $(y,\, z)\,\in\, E_G\times G$, the point
$q_0(y,\, z)\, \in\, E_G$ will be denoted by $yz$. Given a holomorphic principal $G$--bundle
$E_G$ as above, consider the quotient of $E_G\times G$, where two points $(y,\, z)$ and $(y_1,\,
z_1)$ of $E_G\times G$ are identified if there is some $g_0\, \in\, G$ such that $y_1\,=\,
yg_0$ and $z_1\,=\, g^{-1}_0zg_0$. Let
\begin{equation}\label{e0}
f_0\, :\, E_G\times G\, \longrightarrow\, \text{Ad}(E_G)\,:=\, (E_G\times G)/\sim 
\end{equation}
be this quotient. Each fiber of the projection
$$
p\, :\, \text{Ad}(E_G)\, \longrightarrow\, M\, ,\ \ (y,\, z)\, \longmapsto\, p_0(y)\, ,\ \
(y,\, z)\,\in\, E_G\times G
$$
is a group isomorphic to $G$, where the group operation is given by
$(y,\, z)\cdot (y,\, z')\,=\, (y,\, zz')$ (it is straightforward to check that the
group operation is well-defined, meaning it is independent of $y$); note that the map
$(y,\, z)\, \longmapsto\, z$ is an isomorphism of $\text{Ad}(E_G)_x\,:=\, p^{-1}(x)$ with $G$.
For any $x\, \in\, M$, the above
isomorphism between $\text{Ad}(E_G)_x$ and $G$ depends on the choice of
the point $y\, \in \, p^{-1}_0(x)$. However, for two different choices of $y$, the corresponding
isomorphisms differ by an inner automorphism of $G$. In other words, $\text{Ad}(E_G)_x$ and
$G$ are identified uniquely up to an inner automorphism. This $\text{Ad}(E_G)$ is called the
adjoint bundle for $E_G$.

The adjoint action of any $z\, \in\, G$ on $G$ fixes the subgroup $Z_0(G)$ pointwise. Therefore,
$M\times Z_0(G)\, \longrightarrow\, M$ is a subbundle of $\text{Ad}(E_G)$.

Let
\begin{equation}\label{e2}
{\mathbb T}\, \subset\, \text{Ad}(E_G)\, \stackrel{p}{\longrightarrow}\, M
\end{equation}
be a holomorphic sub-fiber bundle such that
\begin{itemize}
\item $M\times Z_0(G)\, \subset\, {\mathbb T}$, and

\item for every point $x\, \in\, M$, the
fiber
$$
{\mathbb T}_x\,:=\, {\mathbb T}\cap \text{Ad}(E_G)_x\,\subset\, \text{Ad}(E_G)_x
$$
is a torus (it need not be a maximal torus of $\text{Ad}(E_G)_x$).
\end{itemize}

Take any point $x\, \in\,
M$. Since $\text{Ad}(E_G)_x$ is identified with $G$ up to an inner automorphism, the torus
${\mathbb T}_x\,\subset\, \text{Ad}(E_G)_x$ determines a conjugacy class of tori in
$G$. From the rigidity of tori in $G$ it follows that this conjugacy class of tori
in $G$ is independent of the choice of the point $x\, \in\, M$. Note that any torus in $G$
is conjugate to a sub-torus of $T_G$, and the space of sub-tori in $T_G$ is a discrete (countable)
set. Fix a torus
\begin{equation}\label{e3}
T\, \subset\, T_G\, \subset\, G
\end{equation}
in the conjugacy class determined by $\mathbb T$. Note that we have $$Z_0(G)\, \subset\, T\, ,$$
because $M\times Z_0(G)\, \subset\, {\mathbb T}$.

Let
\begin{equation}\label{e4}
N \, :=\, N_G(T) \subset\, G
\end{equation}
be the normalizer in $G$ of the subgroup $T$ in \eqref{e3}. The connected component of $N$
\begin{equation}\label{e5}
C\, :=\, N_0\, \subset\, N
\end{equation}
containing the identity element is in fact the centralizer of $T$ in $G$. This component $C$ of $N$
is a Levi factor of a parabolic subgroup of $G$ \cite[\S~3]{SpSt} (see also
\cite[p.~110, Theorem 6.4.7(i)]{Sp}). 
The quotient $N/C$ is the Weyl group associated to $T\, \subset\, G$; it
is a finite group. The normalizer of $C$ in $G$ actually coincides with $N$.

Now let
\begin{equation}\label{en}
E_N\, \subset\, E_G
\end{equation}
be the unique largest subset such that
\begin{equation}\label{e6}
E_N\times T\, \subset \, f^{-1}_0({\mathbb T}) \, \subset\, E_G\, ,
\end{equation}
where $f_0$, $\mathbb T$ and $T$ are as in \eqref{e0}, \eqref{e2} and \eqref{e3} respectively. Let
\begin{equation}\label{p02}
p_0\, :\, E_N\, \longrightarrow\, M
\end{equation}
be the restriction of the projection in \eqref{p0}; this repetition of notation should not
cause any confusion. Clearly, $E_N$ is a complex manifold, and the projection $p_0$
in \eqref{p02} is holomorphic because the projection in \eqref{p0} is holomorphic.

We will prove that $E_N$ is a holomorphic principal $N$--bundle over $M$, where $N$ is
constructed in \eqref{e4}. For this first note that for any $y\, \in\, E_N$ and $g_0\, \in\, G$,
we have
\begin{equation}\label{enc}
f_0(yg_0,\, t)\, =\, f_0(y,\, g_0tg^{-1}_0)
\end{equation}
(see the construction of $f_0$ in \eqref{e0}). Using \eqref{enc} we will show that $yg_0\, \in\, E_N$
if and only if $g_0\, \in\, N$, which would imply that $E_N$ is a holomorphic principal
$N$--bundle over $M$. To see that $yg_0\, \in\, E_N$ for all $g_0\, \in\, N$,
first note that if $g_0\, \in\, N$, then
$g_0tg^{-1}_0\, \in\, T$ whenever $t\, \in\, T$. Therefore, in view of
\eqref{enc}, the given condition that
$\{y\}\times T\, \subset\, f^{-1}_0({\mathbb T})$ immediately implies that
$\{yg_0\}\times T\, \subset\, f^{-1}_0({\mathbb T})$ if $g_0\, \in\, N$. So we have
$yg_0\, \in\, E_N$ for all $g_0\, \in\, N$.

To prove the converse, note that
the given condition that $\{y\}\times T\, \subset\, f^{-1}_0({\mathbb T})$ implies that
$f_0(\{y\}\times T)\,=\, {\mathbb T}_{p_0(y)}$, because $T$ is identified with 
${\mathbb T}_{p_0(y)}$ up to an inner automorphism. Therefore, if $yg_0\, \in\, E_N$,
then from \eqref{enc} it follows that $g_0tg^{-1}_0\, \in\, T$ for all $t\, \in\, T$. This
implies that $g_0\, \in\, N$ if $yg_0\, \in\, E_N$.

{}From \eqref{en} we conclude that the principal $N$--bundle
$E_N$ is a holomorphic reduction of structure group
of $E_G$ to the subgroup $N$ in \eqref{e4}. Equivalently, $E_G$ coincides with the holomorphic
principal $G$--bundle on $M$ obtained by extending the structure group of the holomorphic
principal $N$--bundle $E_N$ using the inclusion of $N$ in $G$.

For notational convenience, the Weyl group $N/C$ in \eqref{e5} will be denoted by $W$. Let
\begin{equation}\label{e7}
E_W\, :=\, E_N/C \,=\, E_N\times^N W\, \stackrel{\phi}{\longrightarrow}\, M
\end{equation}
be the principal $W$--bundle over $M$ obtained by extending the structure group of
the principal $N$--bundle $E_N$ using the
quotient map $N\, \longrightarrow\, N/C \,=\, W$. So $\phi$ in \eqref{e7} is an \'etale
Galois covering with Galois group $\text{Gal}(\phi)\,=\, W$. This $E_W$ need not be connected.

\begin{lemma}\label{lem1}
The pulled back principal $N$--bundle $\phi^*E_N\, \longrightarrow\, E_W$ has a tautological
reduction of structure group to the subgroup $C\, \subset\, N$ in \eqref{e5}, where $\phi$
is the projection in \eqref{e7}.
\end{lemma}

\begin{proof}
Consider the quotient map $E_N\, \longrightarrow\, E_N/C \,=\, E_W$. This makes $E_N$
a holomorphic principal $C$--bundle over $E_W$. We will denote by $E'_C$
this holomorphic principal $C$--bundle over $E_W$. The identity map $E_N\,=\, E'_C
\longrightarrow\, E_N$ is $C$--equivariant. For any $z\, \in\, E_W$, the $C$--equivariant
map $(E'_C)_z \, \longrightarrow\, (E_N)_{\phi(z)}$, obtained by restricting the above identity
map of $E_N$, is evidently an embedding. Therefore, $E'_C$ is a holomorphic reduction of
structure group of the principal $N$--bundle $\phi^*E_N$ to the subgroup $C\, \subset\, N$.
\end{proof}

Since $E_N$ is a holomorphic reduction of structure group of $E_G$ to $N$, Lemma
\ref{lem1} implies that $E'_C$ is also a holomorphic
reduction of structure group of $\phi^* E_G$ to the subgroup $C\, \subset\, G$.

Let
\begin{equation}\label{psi}
\psi\, :\, E'_C\,=\, E_N\, \longrightarrow\, E_W\, :=\, E_N/C
\end{equation}
be the quotient map.

We note that the group $N$ acts holomorphically on $E'_C$, because $E'_C\,=\, E_N$; for
this action $\tau\, :\, E'_C\times N \, \longrightarrow \, E'_C$
of $N$ on $E'_C$, the following diagram of maps is evidently commutative:
\begin{equation}\label{e8}
\begin{matrix}
E'_C\times N & \stackrel{\tau}{\longrightarrow} & E'_C\\
\Big\downarrow && ~\Big\downarrow\psi \\
E_W\times W := (E'_C/C)\times (N/C)& \longrightarrow & E_W
\end{matrix}
\end{equation}
where $E'_C\times N\, \longrightarrow\, (E'_C/C)\times (N/C)\,=\, (E'_C\times N)/(C\times C)$
is the natural quotient map, and $E_W\times W\, \longrightarrow\, E_W$ is the action of $W$
on the principal $W$--bundle $E_W$, while $\psi$ is the map in \eqref{psi}.

\subsection{A torus subbundle from a Levi reduction}\label{se2.2}

We will now describe a reverse of the construction done in Section \ref{se2.1}.

As before, let $E_G$ be a holomorphic principal $G$--bundle over $X$.
Take a torus $$T\, \subset\, T_G\,\subset\, B_G$$
such that $Z_0(G)\, \subset\, T$. As before, the normalizer (respectively,
centralizer) of $T$ in $G$ will be denoted by $N$ (respectively, $C$). Denote the Weyl
group $N/C$ by $W$.

Let $\phi\, :\, E_W\, \longrightarrow\, W$ be a principal $W$--bundle. Assume that the
holomorphic principal $G$--bundle $\phi^*E_G\, \longrightarrow\, E_W$ has a holomorphic
reduction of structure group to the subgroup $C\, \subset\, G$
\begin{equation}\label{f1}
E'_C\, \subset\, \phi^*E_G
\end{equation}
such that
\begin{enumerate}
\item the principal $C$--bundle $E'_C$ is equipped with a holomorphic action
of the complex group $N$
$$
\tau\, :\, E'_C\times N\, \longrightarrow\, E'_C
$$
that extends the natural action of the subgroup $C$ on the principal $C$--bundle $E'_C$, and

\item the diagram of maps
\begin{equation}\label{e9}
\begin{matrix}
E'_C\times N & \stackrel{\tau}{\longrightarrow} & E'_C\\
\Big\downarrow && ~\Big\downarrow\psi \\
E_W\times W := (E'_C/C)\times (N/C)& \longrightarrow & E_W
\end{matrix}
\end{equation}
is commutative, where $E'_C\times N\, \longrightarrow\, (E'_C/C)\times (N/C)\,=\,
(E'_C\times N)/(C\times C)$ is the natural quotient map, and $\psi$ is the
projection of $E'_C$ to $E_W$, so $\psi$ coincides with
the restriction to $E'_C$ of the pullback $\phi^*p_0$,
where $p_0$ is the projection in \eqref{p0}; note that \eqref{e9} is similar to \eqref{e8}.
\end{enumerate}

\begin{lemma}\label{lem2}
The action of $W$ on the principal $W$--bundle $E_W$ has a canonical lift to an action of
$W$ on the adjoint bundle ${\rm Ad}(E'_C)\,=\, E'_C\times^C C\,\longrightarrow \, E_W$
for $E'_C$. This action of $W$ on ${\rm Ad}(E'_C)$ preserves the group structure of the fibers of ${\rm Ad}(E'_C)$.
\end{lemma}

\begin{proof}
The adjoint action of $N$ on itself preserves $C$ because it is the connected component of
$N$ containing the identity element.
We recall that $\text{Ad}(E'_C)$ is the quotient of $E'_C\times C$ where two elements $(y,\, z),\,
(y_1,\, z_1)\, \in\, E'_C\times C$ are identified if there is an element $c\, \in\, C$ such that
$y_1\,=\, yc$ and $z_1\,=\, c^{-1}zc$. 

Consider the following action of $N$ on $E'_C\times C$: the action of any $n\, \in\, N$ sends any
$(y,\, z)\, \in\, E'_C\times C$ to $(\tau(y,\, n),\, n^{-1}zn)$, where $\tau$ is the action
of $N$ in \eqref{e9} that extends the action of $C$ on $E'_C$. So for any $c\, \in\, C$,
the action of $n$ sends $(yc,\, c^{-1}zc)$ to $(\tau(yc,\, n),\, n^{-1}c^{-1}zcn)\,=\,
(\tau(y,\, cn),\, (cn)^{-1}zcn)$ (recall that the restriction of the action $\tau$ to
$C\, \subset\, N$ coincides with the natural action of $C$ on $E'_C$).
The image of $(yc,\, c^{-1}zc)$ (respectively,
$(\tau(y,\, cn),\, (cn)^{-1}zcn)$) in the quotient space $\text{Ad}(E'_C)$ of $E'_C\times C$
coincides with the image of $(y,\, z)$ (respectively, $(yc,\, c^{-1}zc)$).
Therefore, the above action of $N$ on $E'_C\times C$ produces an action of $N$ on the
quotient manifold $\text{Ad}(E'_C)$. Next note that if $n\, \in\, C$, then
the image of $(\tau(y,\, n),\, n^{-1}zn)$ in $\text{Ad}(E'_C)$ coincides with the image of
$(y,\, z)$ in $\text{Ad}(E'_C)$. Consequently, the above action of $N$ on
$\text{Ad}(E'_C)$ produces an action of $W\,=\, N/C$ on $\text{Ad}(E'_C)$.

The above action of $W$ on
$\text{Ad}(E'_C)$ preserves the group structure of the fibers of ${\rm Ad}(E'_C)$, because
the adjoint action of a group on itself preserves the group structure.
\end{proof}

Consider the action $\tau$ of $N$ on $E'_C$ in \eqref{e9}. From \eqref{e9} it follows that this 
action and the composition of maps $E'_C\, \stackrel{\psi}{\longrightarrow}\, E_W \, 
\stackrel{\phi}\longrightarrow\, M$ together define a holomorphic principal $N$--bundle over 
$M$. This holomorphic principal $N$--bundle over $M$ will be denoted by $E_N$.

\begin{lemma}\label{lem3}
The holomorphic principal $G$--bundle over $M$, obtained by extending the structure group
of the above defined principal $N$--bundle $E_N$ using the inclusion of $N$ in $G$, is
identified with $E_G$.
\end{lemma}

\begin{proof}
Let $E_N\, \longrightarrow\, E_G$ be the composition of the natural map $\phi^*E_G\,\longrightarrow\,
E_G$ with the inclusion $E_N\,=\, E'_C\, \hookrightarrow\, \phi^*E_G$ in \eqref{f1}. This map
$E_N\, \longrightarrow\, E_G$ is clearly $N$--equivariant. This implies that $E_N$ is a reduction
of structure group of the principal $G$--bundle $E_G$ to the subgroup $N$. In other words, $E_G$ is
identified with the holomorphic principal $G$--bundle over $M$ obtained by extending the structure
group of the principal $N$--bundle $E_N$ using the inclusion of $N$ in $G$.
\end{proof}

Let
\begin{equation}\label{tp}
\text{Ad}(E'_C)\, \supset\, {\mathbb T}' \,\longrightarrow \, E_W
\end{equation}
be the bundle of connected components of centers
containing identity element; so for any $y\, \in\, E_W$, the fiber $({\mathbb T}')_y$ of
$\mathbb T$ is the connected component of the center of the group $\text{Ad}(E'_C)_y$ containing the
identity element. Note that ${\mathbb T}'$ is identified with the trivial bundle $E_W\times T
\,\longrightarrow \, E_W$, because $T$ is the connected component, containing the
identity element, of the center of $C$; the identification between $E_W\times T$ and
${\mathbb T}'$ sends any $(y,\, z)\, \in\, E_W\times T$ to the equivalence class of
$(y',\, z)$ in the quotient $\text{Ad}(E'_C)$ of $E'_C\times C$, where $y'$ is any element of
the fiber $(E'_C)_y$ (this equivalence class does not depend on the choice of the point
$y'$ in $(E'_C)_y$). The action of $W$ on $\text{Ad}(E'_C)$ in Lemma \ref{lem2}
preserves ${\mathbb T}'$ because the action of $W$ preserves the group structure of the fibers of 
${\rm Ad}(E'_C)$ (hence the centers of fibers are preserved implying that their
connected components containing identity element are also preserved). Consequently, we have a
torus subbundle
\begin{equation}\label{e10}
{\mathbb T}\, :=\, {\mathbb T}'/W\, \subset\, {\rm Ad}(E'_C)/W \, \subset\,
{\rm Ad}(\phi^* E_G)/W\,=\, \text{Ad}(E_G)\, .
\end{equation}

The construction of ${\mathbb T}$ in \eqref{e10} from $E'_C$ in \eqref{f1} is the
inverse of the construction of $E'_C$ in Lemma \ref{lem1} from ${\mathbb T}$
in \eqref{e2}. More precisely, starting with ${\mathbb T}$ in \eqref{e2}, construct
$E'_C$ as in Lemma \ref{lem1}. Now set $E'_C$ in \eqref{f1} to be this
holomorphic principal $C$--bundle constructed from ${\mathbb T}$ in \eqref{e2}. Then
the torus bundle ${\mathbb T}$ constructed in \eqref{e10} coincides with the torus
bundle in \eqref{e2} that we started with.

Conversely, start with $E'_C$ in \eqref{f1} and construct $\mathbb T$ as in \eqref{e10}.
Setting this to be the torus bundle in \eqref{e2}, the principal $C$--bundle
constructed in Lemma \ref{lem1} coincides with $E'_C$ in \eqref{f1} that we started with.

\begin{remark}\label{rem1}
Although the principal $C$--bundle $E'_C$ in Lemma \ref{lem2} does not, in general,
descend to $M$, note that the adjoint bundle $\text{Ad}(E'_C)$, being $W$ equivariant
(see Lemma \ref{lem2}), descends to $M$ as
a subbundle of the adjoint bundle $\text{Ad}(E_N)$ of $E_N$ in Lemma \ref{lem3}. For every
point $x\, \in\, M$, the fiber of $\text{Ad}(E'_C)/W$ over $x$ is the connected component
of $\text{Ad}(E_N)_x$ containing the identity element. Recall that for any $z\, \in\, E_W$,
the fiber ${\mathbb T}'_z$ in \eqref{tp} is the connected component of the center of the group
$\text{Ad}(E'_C)_z$ containing the identity element. Consequently, for any $x\, \in\, X$, the fiber
${\mathbb T}_x$ in \eqref{e10} is the connected component, containing the
identity element, of the center of the fiber of $\text{Ad}(E'_C)/W$ over $x$.
\end{remark}

Combining the results of Section \ref{se2.1} and Section \ref{se2.2}, we have the following:

\begin{theorem}\label{thmse1}
Let $E_G$ be a holomorphic principal $G$--bundle over $M$ and $T\, \subset\, G$ a torus
containing $Z_0(G)$. The normalizer (respectively, centralizer) of $T$ in $G$ will be denoted by
$N$ (respectively, $C$), while the Weyl group $N/C$ will be denoted by $W$.
There is a natural bijective correspondence between the following two:
\begin{enumerate}
\item Torus subbundles $\mathbb T$ of ${\rm Ad}(E_G)$ such that for some
(hence every) $x\, \in\, M$, the 
fiber ${\mathbb T}_x$ lies in the conjugacy class of tori in ${\rm Ad}(E_G)$ determined by $T$.

\item Quadruples of the form $(E_W,\, \phi,\, E'_C,\, \tau)$, where $\phi\, :\, E_W\,
\longrightarrow\, M$ is a principal $W$--bundle, $\phi^*E_G\, \supset\, E'_C\,
\stackrel{\psi}{\longrightarrow}\, E_W$ is a holomorphic
reduction of structure group of $\phi^* E_G$ to $C$, and
$$
\tau\,:\, E'_C\times N \, \longrightarrow\, E'_C
$$
is a holomorphic action of $N$ on $E'_C$ extending the natural action of $C$ on $E'_C$,
such that the diagram of maps
$$
\begin{matrix}
E'_C\times N & \stackrel{\tau}{\longrightarrow} & E'_C\\
\Big\downarrow && ~\Big\downarrow\psi \\
E_W\times W := (E'_C/C)\times (N/C)& \longrightarrow & E_W
\end{matrix}
$$
is commutative, where $E'_C\times N\, \longrightarrow\, (E'_C/C)\times (N/C)\,=\,
(E'_C\times N)/(C\times C)$ is the quotient map.
\end{enumerate}
\end{theorem}

\subsection{A tautological connection on a torus bundle}

We return to the set-up of Section \ref{se2.1}. As in \eqref{e2}, let
$$
{\mathbb T}\, \subset\, \text{Ad}(E_G)\, \stackrel{p}{\longrightarrow}\, M
$$
be a holomorphic sub-fiber bundle containing $M\times Z_0(G)$ such that
for every point $x\, \in\, M$, the fiber
$$
{\mathbb T}_x\,:=\, {\mathbb T}\cap p^{-1}(x)\,\subset\, \text{Ad}(E_G)_x
$$
is a torus.

A flat connection on the fiber bundle $\mathbb T$ is said to be compatible with the
group structure of the fibers of $\mathbb T$ if for any two locally defined flat sections
$s$ and $t$ of $\mathbb T$, defined over an open subset $U\, \subset\, M$, the section
$s\cdot t$ of ${\mathbb T}\vert_U$ is again flat.

\begin{proposition}\label{prop1}
There is a tautological flat holomorphic connection on $\mathbb T$ which is compatible with the
group structure of the fibers of $\mathbb T$.
\end{proposition}

\begin{proof}
Consider the \'etale Galois covering $\phi\, :\, E_W\, \longrightarrow\, M$ is
\eqref{e7}. Let $$E'_C\, \subset\, \phi^* E_G$$ be the holomorphic reduction of structure
group constructed in Lemma \ref{lem1}. The adjoint action of $C$ on the torus
$T$ in \eqref{e3} is trivial, because $T$ is contained in the center of $C$. Therefore, $T$ defines
a trivial subbundle
\begin{equation}\label{se2l}
E_W\times T\, \subset\, \text{Ad}(E'_C)\, \subset\, \text{Ad}(\phi^*E_N)\,=\,
\phi^*\text{Ad}(E_N)\, \subset \, \phi^*\text{Ad}(E_G)\,=\,
\text{Ad}(\phi^*E_G)\, ,
\end{equation}
where $E_N$ is the principal $N$--bundle constructed in \eqref{en}. Note that $E_W\times T\,
\subset\, \text{Ad}(E'_C)$ is the bundle consisting of connected components of the centers,
containing the identity element, of the fibers of $\text{Ad}(E'_C)$, and $\text{Ad}(E'_C)$ is the
connected component of $\text{Ad}(\phi^*E_N)\,=\, \phi^*\text{Ad}(E_N)$ containing the
section given by the identity elements of the fibers.

Let $D_0$ denote the trivial (flat) connection on the trivial bundle $E_W\times 
T\,\longrightarrow\, E_W$ in \eqref{se2l}. The tautological action of the Galois group 
$W\,=\, \text{Gal}(\phi)$ on $\phi^*\text{Ad}(E_N)$ evidently preserves the subbundle 
$E_W\times T \, \subset\, \phi^*\text{Ad}(E_N)$. The resulting action of $W$ on $E_W\times 
T$ clearly preserves the above trivial connection $D_0$ on $E_W\times T\,\longrightarrow\, 
E_W$.

Since the connection $D_0$ is preserved by the action of $W\,=\, \text{Gal}(\phi)$, it
produces a flat holomorphic connection on the subbundle $(E_W\times T)/W$
$$
\text{Ad}(E_N)\,=\, (\phi^*\text{Ad}(E_N))/W \, \supset\, (E_W\times T)/W \,
\longrightarrow\, E_W/W\,=\, M
$$
of $\text{Ad}(E_N)$. But this subbundle $(E_W\times T)/W$ is identified with $\mathbb T$, because
$E_W\times T\, \subset\, \phi^*\text{Ad}(E_G)$ in \eqref{se2l} coincides with $\phi^*{\mathbb T}
\, \subset\, \phi^*\text{Ad}(E_G)$ (see \eqref{e10}; recall that ${\mathbb T}'$ in
\eqref{e10} is identified with $E_W\times T$).
\end{proof}

\section{Torus subbundles and connection}\label{se3}

\subsection{Connections on a principal bundle}\label{se3.1}

Let $H$ be a complex Lie group. The Lie algebra of $H$ will be denoted by $\mathfrak h$.
Let $$\beta\, :\, E_H\, \longrightarrow\, M$$ be a holomorphic principal $H$--bundle on $M$. The
holomorphic tangent bundles of $M$ and $E_H$ will be denoted by $TM$ and $TE_H$ respectively. The
quotient
$$
\text{At}(E_H)\, :=\, (TE_H)/H\, \longrightarrow\, E_H/H \,=\, M
$$
is a holomorphic vector bundle which is called the \textit{Atiyah bundle} for $E_H$. The
differential $d\beta\, :\, TE_H\, \longrightarrow\, \beta^*TM$ for the above
projection $\beta$, being $H$--invariant, produces a surjective homomorphism
$$
d'\beta\, :\, \text{At}(E_H)\, \longrightarrow\, TM\, .
$$
The kernel of $d'\beta$ is the relative tangent bundle for $\beta$ and it is identified with
the adjoint vector bundle $\text{ad}(E_H)\,=\, \text{kernel}(d\beta)/H$.
We recall that $\text{ad}(E_H)$ is the quotient of $E_H\times\mathfrak h$, where two points $(y,\, v)$
and $(y_1,\, v_1)$ of $E_H\times\mathfrak h$ are identified if there is some $h_0\, \in\, H$ such that
$y_1\,=\, yh_0$ and $v_1\,=\, \text{Ad}(h^{-1}_0)(v)$.
Therefore, $\text{ad}(E_H)\, \longrightarrow\, M$ is the Lie algebra bundle for the bundle
$\text{Ad}(E_H)$ of Lie groups on $M$. We have the short exact sequence of holomorphic vector bundles on $M$
$$
0\, \longrightarrow\, \text{ad}(E_H)\, \longrightarrow\, \text{At}(E_H) 
\, \stackrel{d'\beta}{\longrightarrow}\, TM \, \longrightarrow\, 0\, ,
$$
which is know as the Atiyah exact sequence. A complex connection on $E_H$ is a $C^\infty$
homomorphism of vector bundles
$$
D_0\, :\, TM\, \longrightarrow\, \text{At}(E_H)
$$
such that $(d'\beta)\circ D_0\,=\, \text{Id}_{TM}$ (see \cite{At}).
A holomorphic connection on $E_H$ is a holomorphic
homomorphism of vector bundles
$$
D_0\, :\, TM\, \longrightarrow\, \text{At}(E_H)
$$
such that $(d'\beta)\circ D_0\,=\, \text{Id}_{TM}$.

Let $H'$ be a complex Lie group and $\eta\, :\, H\, \longrightarrow\, H'$ a holomorphic
homomorphism of Lie groups. Let
$$
E_{H'}\, :=\, E_H\times^\eta H\, \longrightarrow\, M
$$
be the holomorphic principal $H'$--bundle over $M$ obtained by extending the structure
group of the holomorphic principal $H$--bundle $E_H$ using the above homomorphism $\eta$. Recall
that $E_{H'}$ is the quotient of $E_H\times H'$ where any two points $(y,\, z)$ and
$(y_1,\, z_1)$ of $E_H\times H'$ are identified if there is an element $h\, \in\, H$ such that
$y_1\,=\, yh$ and $z_1\,=\, \eta(h^{-1})z$. So sending any $y\, \in\, E_H$ to the equivalence
class of $(y,\, e)$, where $e\, \in\, H'$ is the identity element, we get a holomorphic map
$$
\eta_0\, :\, E_H\, \longrightarrow\, E_{H'}
$$
which satisfies the equation
\begin{equation}\label{ae}
\eta_0(yh)\,=\, \eta_0(y)\eta(h)
\end{equation}
for all $y\, \in\, E_H$ and $h\, \in\, H$. The action of $H$ on $E_H$ produces an action of
$H$ on the tangent bundle $TE_H$; using $\eta$, the action of $H'$ on the tangent bundle
$TE_{H'}$ produces an action of $H$ on $TE_{H'}$. From \eqref{ae} it follows immediately that
the differential $d\eta_0$ of the map $\eta_0$ is $H$--equivariant. Therefore, $d\eta_0$
induces a homomorphism
$$
\eta_1\, :\, \text{At}(E_H)\, =\, (TE_H)/H\, \longrightarrow\, (TE_{H'})/H'\,=\, \text{At}(E_{H'})\, .
$$
This homomorphism $\eta_1$ satisfies the equation $\eta_1(\text{ad}(E_H))\, \subset\,
\text{ad}(E_{H'})$; in other words, we have a commutative diagram of holomorphic homomorphisms
of vector bundles
$$
\begin{matrix}
0 & \longrightarrow & \text{ad}(E_H) & \longrightarrow & \text{At}(E_H) 
& \stackrel{d'\beta}{\longrightarrow}& TM & \longrightarrow & 0\\
&&\Big\downarrow && \,\,\,\,\,\, \Big\downarrow \eta_1 &&\Vert\\
0 & \longrightarrow & \text{ad}(E_{H'}) & \longrightarrow & \text{At}(E_{H'}) 
& \longrightarrow & TM & \longrightarrow & 0
\end{matrix}
$$
where the horizontal sequences are the Atiyah exact sequences.
Therefore, if $D_0\, :\, TM\, \longrightarrow\, \text{At}(E_H)$ is a complex connection
on $E_H$, then
$$
\eta_1 \circ D_0\, :\, TM\, \longrightarrow\, \text{At}(E_{H'})
$$
is a complex connection on $E_{H'}$; this $\eta_1 \circ D_0$ is called the
connection on $E_{H'}$ induced by the connection $D_0$ on $E_H$. If the connection $D_0$
is holomorphic, then the induced connection $\eta_1 \circ D_0$ is clearly holomorphic also.

\subsection{Criterion for induced connection}

We continue with the set-up of Section \ref{se2.1}. Take a torus subbundle
$$
{\mathbb T}\, \subset\, \text{Ad}(E_G)\, \stackrel{p}{\longrightarrow}\, M
$$
as in \eqref{e2}. Fix $T$ as in \eqref{e3}, and construct $N$ as in \eqref{e4}.
Let $E_N\, \subset\, E_G$ be the holomorphic reduction of structure group to $N$
constructed in \eqref{en}.

Let $D_0$ be a complex connection on $E_G$. Our aim in this section is to establish a
necessary and sufficient condition for $D_0$ to be induced by a complex connection on $E_N$.

The connection $D_0$ on $E_G$ induces a complex connection on every holomorphic
fiber bundle $E_G(F)$ associated to $E_G$ for a holomorphic action of $G$ on a complex
manifold $F$. In particular, $D_0$ induces a complex connection on the bundle
$\text{Ad}(E_G)$ associated to $E_G$ for the adjoint action of $G$ on itself. Let $D_1$
denote the complex connection on $\text{Ad}(E_G)$ induced by $D_0$.

\begin{theorem}\label{thm1}
The connection $D_0$ on $E_G$ is induced by a complex connection on the principal
$N$--bundle $E_N$ if and only if the induced connection $D_1$ on ${\rm Ad}(E_G)$ preserves
the torus subbundle ${\mathbb T}\, \subset\, {\rm Ad}(E_G)$.
\end{theorem}

\begin{proof}
First assume that there is a complex connection $D_2$ on the principal
$N$--bundle $E_N$ such that the connection on $E_G$ induced by $D_2$ coincides with $D_0$.
Let $D_3$ be the complex connection on ${\rm Ad}(E_N)$ induced by $D_2$. Let
${\mathcal G}\, \subset\, {\rm Ad}(E_N)$ be the sub-fiber bundle whose fiber over any
$x\, \in\, M$ is the connected component of ${\rm Ad}(E_N)_x$ containing the identity element.
So $\mathcal G$ coincides with $\text{Ad}(E'_C)/W$ in Remark \ref{rem1}. The connection $D_3$
on ${\rm Ad}(E_N)$ clearly preserves ${\mathcal G}$. Let $D_4$ denote the complex connection on
${\mathcal G}$ given by $D_3$.

Since $D_0$ is induced by $D_2$, it follows that the
connection $D_1$ on ${\rm Ad}(E_G)$ preserves the torus subbundle ${\mathbb T}$
if and only if the connection $D_4$ on $\mathcal G$ preserves $\mathbb T$.

We recall that the torus subbundle ${\mathbb T}\, \subset\, \mathcal G$ is the bundle of 
connected components of centers, containing the identity element, of the fibers of
$\mathcal G$ (see Remark \ref{rem1}). From this it can be deduced that
the connection on $D_4$ on $\mathcal G$ preserves $\mathbb T$. To prove this, it is
convenient to switch to the Lie algebra bundles from the Lie group bundles, because
it is easier the work with connections on vector bundles.

Let $\widetilde{\mathbb T}$ (respectively, $\widetilde{\mathcal G}$) be the Lie algebra
bundle over $M$ corresponding to the Lie group bundle ${\mathbb T}$ (respectively, ${\mathcal G}$).
Note that $\widetilde{\mathcal G}$ coincides with $\text{ad}(E_N)$, because $N/C$ is a finite
group so $\text{Lie}(N)\,=\, \text{Lie}(C)$. Since ${\mathbb T}\, \subset\, \mathcal G$ is the
bundle of connected components of centers, containing the identity element, of the fibers of
$\mathcal G$, it follows immediately
that for any $x\, \in\, M$, the fiber $\widetilde{\mathbb T}_x$ is the center of the
Lie algebra $\widetilde{\mathcal G}_x$. Let $D'_4$ be the complex connection on the
vector bundle $\widetilde{\mathcal G}$ given by the connection $D_4$ on $\mathcal G$; note
that $D'_4$ coincides with the connection on $\text{ad}(E_N)$ induced by the connection
$D_2$ on $E_N$.
If $s$ and $t$ are locally defined holomorphic sections of
$\widetilde{\mathcal G}$ defined over an open subset $U\, \subset\, M$, then we have
$$
D'_4([s, \,t])\,=\, [D'_4(s),\, t] + [s,\, D'_4(t)]\, ,
$$
because $D'_4$ is compatible with the Lie algebra structure of the fibers of $\widetilde{\mathcal G}$.
Now if $s$ is a section of the subbundle $\widetilde{\mathbb T}$, then
$$
[s,\, t]\,=\, 0\,=\, [s,\, D'_4(t)]\, .
$$
Therefore, we conclude that $[D'_4(s),\, t] \,=\, 0$ if $s$ is a section of
$\widetilde{\mathbb T}\vert_U$. Hence $D'_4(s)$ is a $C^\infty$ section of
$\widetilde{\mathbb T}\vert_U\otimes \Omega^1_U$
if $s$ is a holomorphic section of $\widetilde{\mathbb T}\vert_U$. Consequently, the connection
$D'_4$ on $\widetilde{\mathcal G}$ preserves the subbundle $\widetilde{\mathbb T}$.
Hence the connection $D_4$ on $\mathcal G$ preserves $\mathbb T$.

To prove the converse, assume that the connection $D_0$ on $E_G$ has the following property:
the connection $D_1$ on ${\rm Ad}(E_G)$ induced by $D_0$ preserves
the torus subbundle ${\mathbb T}\, \subset\, {\rm Ad}(E_G)$.

Recall that ${\rm ad}(E_G)$ is the Lie algebra bundle on $M$ corresponding to the 
bundle ${\rm Ad}(E_G)$ of groups. Let $D'_1$ be the complex connection on the vector bundle 
${\rm ad}(E_G)$ induced by $D_1$.
As before, $\widetilde{\mathbb T}$ denotes the Lie algebra bundle on $M$ corresponding to 
$\mathbb T$, so $\widetilde{\mathbb T}$ is an abelian subalgebra bundle of
${\rm ad}(E_G)$. The given condition that the connection $D_1$ on ${\rm Ad}(E_G)$ preserves
$\mathbb T$ immediately implies that the connection $D'_1$ preserves the subbundle
$\widetilde{\mathbb T}\, \subset\, {\rm ad}(E_G)$.

If $s$ and $t$ are locally defined holomorphic 
sections of ${\rm ad}(E_G)$ defined over an open subset $U\, 
\subset\, M$, then we have
$$
D'_1([s, \,t])\,=\, [D'_1(s),\, t] + [s,\, D'_1(t)]\, ;
$$
because $D'_1$ is compatible with the Lie algebra structure of the fibers of ${\rm ad}(E_G)$.
Now if $s$ is a section of the subbundle $\widetilde{\mathbb T}\vert_U$, and
$t$ is a section of $\widetilde{\mathcal G}\vert_U$ (defined earlier), then we have
$$
[s, \,t]\,=\, 0\,=\, [D'_1(s),\, t]\, ;
$$
indeed, $[s, \,t]\,=\, 0$ because $\widetilde{\mathbb T}$ is the bundle the centers of
$\widetilde{\mathcal G}$, and $[D'_1(s),\, t]\,=\, 0$ because
$D'_1$ preserves $\widetilde{\mathcal T}$. Therefore, we have
$$[s,\, D'_1(t)]\,=\, 0$$ if $s$ is a holomorphic
section of the subbundle $\widetilde{\mathbb T}\vert_U$ and $t$ is a holomorphic
section of $\widetilde{\mathcal G}\vert_U$. From this it follows that the subbundle
$\widetilde{\mathcal G}\, \subset\, \text{ad}(E_G)$ is preserved by the connection
$D'_1$ on $\text{ad}(E_G)$. Indeed, as noted before,
${\mathcal G}\,=\, \text{ad}(E_N)\, \subset\, \text{ad}(E_G)$, because the Lie 
algebras $\text{Lie}(N)$ and $\text{Lie}(C)$ coincide (the quotient $N/C$ is a finite group).
Hence for any $x\, \in\, M$, the subalgebra $$\widetilde{\mathcal
G}_x\, \subset\, \text{ad}(E_G)_x$$ is the centralizer of $\widetilde{\mathbb T}_x$.

The normalizer of $\text{Lie}(C)$ in $\text{Lie}(G)$ is $\text{Lie}(C)$ itself; this is
because the normalizer of $C$ in $G$ is $N$, and $\text{Lie}(C)\,=\,\text{Lie}(N)$. Therefore,
from the above observation that $\widetilde{\mathcal G}\,=\, \text{ad}(E_N)$ is preserved by
the connection $D'_1$ on $\text{ad}(E_G)$ it follows that the connection $D_0$ on $E_G$
preserves $E_N$. In other words, the connection $D_0$ on $E_G$ is induced by a connection on $E_N$.
This completes the proof.
\end{proof}

\begin{proposition}\label{prop2}
Let $D_0$ be a complex connection on $E_G$ such that
the induced connection $D_1$ on ${\rm Ad}(E_G)$ preserves
the torus subbundle ${\mathbb T}\, \subset\, {\rm Ad}(E_G)$. Then the connection on
$\mathbb T$ given by $D_1$ coincides with the tautological connection on
$\mathbb T$ in Proposition \ref{prop1}.
\end{proposition}

\begin{proof}
{}From Theorem \ref{thm1} we know that there is a complex connection $D_2$ on
the principal $N$--bundle $E_N$ that induces $D_0$. Consider the pulled back connection
$\phi^*D_2$ on $\phi^*E_N$, where $\phi$ is the projection from $E_W$ in \eqref{e7}.
This connection $\phi^*D_2$ gives a connection
on the principal $C$--bundle $E'_C$ in Lemma \ref{lem1}, because $E'_C$ is a union of some connected
component of $\phi^*E_N$. Let $D''$ be the complex
connection on $\text{Ad}(E'_C)$ induced by this connection on $E'_C$ given by
$\phi^*D_2$.

The adjoint action of $C$ on
$T$ is trivial because $T$ is contained in the center of $T$. Therefore,
\begin{itemize}
\item the subbundle
$$E_W\times T\, \subset\, \text{Ad}(E'_C)$$
in \eqref{se2l} is preserved by the connection $D''$ on $\text{Ad}(E'_C)$, and

\item the connection on $E_W\times T$ given by $D''$ coincides with the trivial
connection of the trivial bundle.
\end{itemize}
Now from the construction of the tautological connection on
$\mathbb T$ in Proposition \ref{prop1} it follows that it coincides with
the connection on $\mathbb T$ given by $D_1$.
\end{proof}

\begin{remark}\label{rem-c}
We use the notation in Theorem \ref{thm1}. Let $D_0$ be a complex connection on $E_G$ such that 
the induced connection $D_1$ on ${\rm Ad}(E_G)$ preserves the torus subbundle ${\mathbb T}\, 
\subset\, {\rm Ad}(E_G)$. Let $D_2$ be the complex connection on $E_N$ inducing $D_0$. If the 
connection $D_2$ is holomorphic, then $D$ is holomorphic, because holomorphic connections induce 
holomorphic connection, as noted in Section \ref{se3.1}. Since $N$ is a subgroup of $G$, and the 
principal $G$--bundle $E_G$ is the extension of structure group of the principal $N$--bundle 
$E_N$ for the inclusion of $N$ in $G$, it follows that the Atiyah bundle $\text{At}(E_N)$ is a 
holomorphic subbundle of $\text{At}(E_G)$. Therefore, if the connection $D$ is holomorphic then 
$D_2$ is also holomorphic.
\end{remark}

\subsection{Hermitian--Einstein connection and Levi reduction}\label{sehe}

In the subsection we assume that $M$ is K\"ahler and it is equipped with a K\"ahler form $\omega$. For any
torsionfree coherent analytic sheaf $W$ on $X$, define
$$
\text{degree}(W)\, :=\, \int_M c_1(\det W)\wedge \omega^{d-1}\, ,
$$
where $d\,=\, \dim_{\mathbb C} M$ and the determinant line bundle $\det W$ is constructed
as in \cite[Ch.~V, \S~6]{Ko}. A torsionfree coherent analytic sheaf $W$ is called \textit{stable}
(respectively, \textit{semistable}) if every coherent analytic subsheaf $F\, \subset\, W$
with $0\, <\, \text{rank}(F)\, <\, \text{rank}(W)$, we have
$$
\frac{\text{degree}(F)}{\text{rank}(F)}\,<\, 
\frac{\text{degree}(W)}{\text{rank}(W)}\ {\rm (respectively,}~
\frac{\text{degree}(F)}{\text{rank}(F)}\,\leq\, 
\frac{\text{degree}(W)}{\text{rank}(W)}{\rm )}
$$
(see \cite[p.~168]{Ko}).
Also, $W$ is called \textit{polystable} if it is semistable and direct sum stable sheaves.

We shall now recall a generalization of these notions to the context of principal bundles
\cite{Ra}, \cite{RS}.

An open dense subset $U$ of $M$ will be called big if the complement $M\setminus U$ is a complex
analytic subspace of $M$ of complex codimension at least two. For a holomorphic line bundle $L$
on a big open subset $\iota \, :\, U\, \hookrightarrow\, M$, the degree of $L$ is defined to
the degree of the direct image $\iota_* L$.

A character $\chi$ of a parabolic subgroup $P\,\subset\, G$ is called \textit{anti-dominant}
if the holomorphic line bundle on $G/P$ associated to $\chi$ is nef. Moreover, if the
associated line bundle is ample, then $\chi$ is called \textit{strictly anti-dominant}.

Let $E_G$ be a holomorphic principal $G$--bundle over $M$, where $G$, as before,
is a connected reductive affine complex algebraic group. It is called \textit{stable}
(respectively, \textit{semistable}) if for all triples of the form $(P,\, E_P,\, \chi)$, where
\begin{itemize}
\item{} $P\, \subset\, G$ is proper (not necessarily maximal) parabolic
subgroup,

\item $E_P\, \subset\, E_G$ is a holomorphic reduction of structure group of $E_G$ to
$P$ over a big open subset $U\, \subset\, M$, and

\item $\chi$ is a strictly anti-dominant character of $P$
which is trivial on the center of $G$,
\end{itemize}
the following holds:
$$
\text{degree}(E_P(\chi))\, > \, 0
$$
(respectively, $\text{degree}(E_P(\chi))\, \geq \, 0$), where $E_P(\chi)$ is the
holomorphic line bundle on $U$ associated to the principal $P$--bundle $E_P$ for the
character $\chi$. (See \cite{Ra}, \cite{RS}, \cite{AB}.)

Let $P$ be a parabolic subgroup of $G$ and $E_P\, \subset\, E_G$
a holomorphic reduction of structure group over $M$ of $E_G$ to the subgroup $P$. Such a reduction
of structure group is called \textit{admissible} if for every character $\chi$ of $P$ trivial
on the center of $G$, the associated holomorphic line bundle $E_P(\chi)$ on $M$ is of
degree zero.

A holomorphic principal $G$--bundle $E_G$ on $M$ is called \textit{polystable} if either $E_G$
is stable or there is parabolic subgroup $P\, \subset\, G$ and a holomorphic reduction of
structure group $E_{L(P)}\, \subset\, E_G$ over $M$ to a Levi factor $L(P)$ of $P$, such that
\begin{itemize}
\item the principal $L(P)$--bundle $E_{L(P)}$ is stable, and

\item the reduction of structure group of $E_G$ to $P$ given by the extension of the structure
group of $E_{L(P)}$ to $P$, using the inclusion of $L(P)$ in $P$, is admissible.
\end{itemize}
(See \cite{RS}, \cite{AB}.)

Fix a maximal compact subgroup
\begin{equation}\label{kg}
K_G\, \subset\, G\, .
\end{equation}
Let $E_{K_G}\, \subset\, E_G$ be a $C^\infty$ reduction of structure group over $M$
of $E_G$ to the subgroup $K_G$. Then there is a unique $C^\infty$ connection $\nabla$ on
$E_{K_G}$ such that the connection on $E_G$ induced by $\nabla$ is a complex connection
\cite[pp.~191--192, Proposition 5]{At}.

Let $z({\mathfrak k})$ denote the center of the Lie algebra of $K_G$. A $C^\infty$ reduction 
of structure group over $M$
$$E_{K_G}\, \subset\, E_G$$
is called a Hermitian--Einstein reduction if the corresponding connection $\nabla$ has the
property that the curvature ${\mathcal K}(\nabla)$ of $\nabla$ satisfies the equation
\begin{equation}\label{ez2}
{\mathcal K}(\nabla)\wedge \omega^{d-1} \,=\, c \omega^d
\end{equation}
for some $c\, \in\, z({\mathfrak k})$, where $d\,=\, \dim_{\mathbb C} M$. If $E_{K_G}$ is
a Hermitian--Einstein reduction, then the connection on $E_G$ induced by the corresponding
connection $\nabla$ on $E_{K_G}$ is called a Hermitian--Einstein connection.

A holomorphic principal $G$--bundle $E_G$ on $M$ admits a Hermitian--Einstein connection if 
and only if $E_G$ is polystable, and furthermore, if $E_G$ is polystable, then it has a unique 
Hermitian--Einstein connection. When $M$ is a complex projective manifold and $G\,=\, 
\text{GL}(r, {\mathbb C})$, this was proved in \cite{Do1}, \cite{Do2}; when $M$ is K\"ahler 
and $G\,=\, \text{GL}(r, {\mathbb C})$, this was proved in \cite{UY}; when $M$ is a complex 
projective manifold and $G$ is an arbitrary complex reductive group, this was proved in 
\cite{RS}; when $M$ is K\"ahler and $G$ is an arbitrary complex reductive group, this was 
proved in \cite{AB}.

Now consider the set-up of Theorem \ref{thmse1}; let $T$, $C$, $N$ and $W$ be as in
Theorem \ref{thmse1}.

Let $E_G$ be a holomorphic principal $G$--bundle over $M$. Let $\mathbb T$
be a torus subbundles of ${\rm Ad}(E_G)$ such that for some
(hence every) $x\, \in\, M$, the
fiber ${\mathbb T}_x$ lies in the conjugacy class of tori in ${\rm Ad}(E_G)$ determined by $T$.
Let $(E_W,\, \phi,\, E'_C,\, \tau)$ be the corresponding quadruple in Theorem \ref{thmse1}.

Equip $E_W$ with the K\"ahler form $\phi^*\omega$. Note that $E_W$ need not be connected.
Take a holomorphic principal $H$--bundle $E_H$ on $E_W$, where $H$ is a connected complex 
reductive affine algebraic group (we do not use the notation $G$ because it
may create confusion as $G$ is used above). We will call $E_H$ to be polystable if the
following two conditions hold:
\begin{enumerate}
\item the restriction of $E_H$ to each connected component of $E_W$ is polystable, and

\item for each character $\chi$ of $H$, the associated holomorphic line bundle $E_H(\chi)$
on $E_W$ has the property that degrees of its restriction to the connected components of
$E_W$ coincide. (If $E_W$ is connected then this condition is vacuously satisfied.)
\end{enumerate}
Fix a maximal compact subgroup $K_H\, \subset\, H$.
Let $z({\mathfrak k}_h)$ denote the center of the Lie algebra of $K_H$. A $C^\infty$ reduction 
of structure group over $E_W$ $$E_{K_H}\, \subset\, E_H$$
will be called a Hermitian--Einstein reduction if the corresponding connection $\nabla$ has the
property that there is an element $c\, \in\, z({\mathfrak k}_h)$ such
that the curvature ${\mathcal K}(\nabla)$ of $\nabla$ satisfies the equation
\begin{equation}\label{ez1}
{\mathcal K}(\nabla)\wedge (\phi^*\omega)^{d-1} \,=\, c (\phi^*\omega)^d\, .
\end{equation}
If $E_{K_H}$ is a Hermitian--Einstein reduction, then the connection on $E_H$ induced by the
corresponding connection $\nabla$ on $E_{K_H}$ will be called a Hermitian--Einstein connection.

If $E_H$ is polystable, then the Hermitian--Einstein connections on the restrictions of $E_H$ to 
the connected components of $E_W$ together produce a Hermitian--Einstein connection on $E_H$. 
The Hermitian--Einstein connection for the restriction of $E_H$ to each component of $E_W$ 
produces an element of $z({\mathfrak k}_h)$ (the element $c$ in \eqref{ez2}). The second 
condition in the definition of polystability for $E_H$ ensures that this element of 
$z({\mathfrak k}_h)$ is independent of the component of $E_W$; the elements of $z({\mathfrak 
k}_h)$ for different connected components of $E_W$ coincide. Furthermore, the 
Hermitian--Einstein connection on $E_H$ is unique. Conversely, if $E_H$ admits a 
Hermitian--Einstein connection, then the restriction of $E_H$ to each connected components of 
$E_W$ is polystable. Since the element of $z({\mathfrak k}_h)$ for the Hermitian--Einstein 
connection (the element $c$ in \eqref{ez1}) does not depend on the component of $E_W$, the 
second condition in the definition of polystability is satisfied for $E_H$.

The intersection $K_C\, :=\, C\bigcap K_G$ is a maximal compact subgroup of $C$. It will be
used for defining the Hermitian--Einstein equation for holomorphic principal $C$--bundles on $E_W$.

\begin{proposition}\label{prop3}
Assume that the principal $G$--bundle $E_G$ on $M$ is polystable. Let $\nabla$
be the Hermitian--Einstein connection on $E_G$. Then the following two
hold:
\begin{enumerate}

\item The principal $C$--bundle $E'_C$ on $E_W$ is polystable.

\item The Hermitian--Einstein connection $\phi^*\nabla$ on $\phi^*E_G$ preserves
the reduction $E'_C$ of structure group of $\phi^*E_G$ to $C$. Furthermore, the connection on $E'_C$
given by $\phi^*\nabla$ is Hermitian--Einstein.
\end{enumerate}
\end{proposition}

\begin{proof}
Recall that $T$ is contained in the center of $C$. So the natural action of $C$
on the principal $C$--bundle $E'_C$ commutes with the action of $T\, \subset\, C$ on $E'_C$. So
$T$ acts on the total space of $E'_C$ preserving its principal $C$--bundle structure.
Recall that ${\mathbb T}'$ in \eqref{tp} is identified with $E_W\times T$. The
above action of $T$ on $E'_C$ produces the identification of $E_W\times T$ with
${\mathbb T}'$. Since $\phi^*E_G$ is identified with
the principal $G$--bundle on $E_W$ obtained by extending the structure group $E'_C$ to the
group $G$, the action of $T$ on $E'_C$ produces an action of $T$ on $\phi^* E_G$. Indeed,
$\phi^*E_G$ is the quotient of $E'_C\times G$ where two elements $(y,\, z)$ and $(y',\, z')$
of $E'_C\times G$ are identified if there is a $g\, \in\, C$ such that
$y'\,=\, yg$ and $z'\,=\, g^{-1}z$. Consider the diagonal action of $T$ on $E'_C\times G$ given by the
above action of $T$ on $E'_C$ and the trivial action of $T$ on $G$. This diagonal action of $T$
on $E'_C\times G$ descends to an action of $T$ on the quotient space $\phi^*E_G$.

Consider the Hermitian--Einstein connection $\phi^*\nabla$ on $\phi^*E_G$. From the uniqueness 
of a Hermitian--Einstein connection it follows immediately that the above action of $T$ on 
$\phi^*E_G$ preserves the Hermitian--Einstein connection $\phi^*\nabla$.
Since ${\mathbb T}\,=\, {\mathbb T}'/W\,=\, (E_W\times T)/W$ (see \eqref{e10}), this implies
that the subbundle ${\mathbb T}\, \subset\, \text{Ad}(E_G)$ is preserved by the connection
on $\text{Ad}(E_G)$ induced by $\nabla$. Hence from Theorem 
\ref{thm1} if follows that the connection $\nabla$ is induced by a complex connection 
$\nabla'$ on the principal $N$--bundle $E_N$.

A connection on $E_N$ gives a connection on the 
principal $C$--bundle $E'_C$; recall that the total spaces of $E_N$ and $E'_C$ coincide. The 
connection on $E'_C$ given by the above connection $\nabla'$ on the principal $N$--bundle $E_N$ will 
be denoted by $\nabla''$. Since the connection $\nabla'$ on $E_N$ induces the connection $\nabla$
on $E_G$, from the definition of $\nabla''$ it follows that the connection $\phi^*\nabla$ on 
$\phi^*E_G$ is induced by $\nabla''$. As the connection $\phi^*\nabla$ is Hermitian--Einstein,
it follows immediately that the inducing connection $\nabla''$ is also Hermitian--Einstein.

Since $\nabla''$ is a Hermitian--Einstein connection on $E'_C$, the principal $C$--bundle
$E'_C$ is polystable.
\end{proof}

We will now prove a converse of Proposition \ref{prop3}. Consider the maximal compact subgroup 
$K_C\, :=\, C\bigcap K_G$ of $C$. Let $z({\mathfrak k}_c)$ be the center of the Lie algebra of 
$K_C$. Note that $z({\mathfrak k})\, \subset\, z({\mathfrak k}_c)$, because the connected 
component of the center of $K_G$, containing the identity element, is contained in $C$.

\begin{proposition}\label{prop4}
Assume that the principal $C$--bundle $E'_C$ over $E_W$ is polystable. Let $\nabla$
be the Hermitian--Einstein connection on $E'_C$. Assume that the element of $z({\mathfrak k}_c)$
given by the curvature of $\nabla$ (the element $c$ in \eqref{ez1}) lies in
the subspace $z({\mathfrak k})$. Then the following two hold:
\begin{enumerate}

\item The principal $G$--bundle $E_G$ on $M$ is polystable.

\item The Hermitian--Einstein connection on $E_G$ is given by $\nabla$.
\end{enumerate}
\end{proposition}

\begin{proof}
Since $\phi^*E_G$ is the extension of structure group of $E'_C$ to $G$ using the inclusion of 
$C$ in $G$, a connection on $E'_C$ induces a connection on $\phi^*E_G$. Let $\nabla'$ be the 
connection of $\phi^*E_G$ induced by the connection $\nabla$ on $E'_C$. Since $\nabla$ satisfies 
the Hermitian--Einstein equation with the element of $z({\mathfrak k}_c)$, given by the 
curvature of $\nabla$, lying in the subspace $z({\mathfrak k})$, it follows immediately that 
$\nabla'$ satisfies the Hermitian--Einstein equation for $\phi^*E_G$.

Consider the natural action of the Galois group $\text{Gal}(\phi)\, =\, W$ on the pulled back 
bundle $\phi^*E_G$. From the uniqueness of the Hermitian--Einstein connection on $\phi^*E_G$ it 
follows immediately that this action of $W$ on $\phi^*E_G$ preserves the Hermitian--Einstein 
connection $\nabla'$. Hence $\nabla'$ descends to a connection on $E_G$; this connection on 
$E_G$ given by $\nabla'$ will be denoted by $\nabla''$. Since $\nabla'$ satisfies the 
Hermitian--Einstein equation, it follows immediately that $\nabla''$ also satisfies the 
Hermitian--Einstein equation.

Since $E_G$ admits a Hermitian--Einstein connection, namely $\nabla''$, the principal
$G$--bundle $E_G$ is polystable.
\end{proof}

\section{Higgs bundles and Levi reduction}\label{se4}

In this section we work with the set-up of Section \ref{se2.1}. Take a torus bundle
$$
{\mathbb T}\, \subset\, \text{Ad}(E_G)\, \stackrel{p}{\longrightarrow}\, M
$$
as in \eqref{e2}. Fix $T$ as in \eqref{e3}, and construct $N$ as in \eqref{e4}.
Let $E_N\, \subset\, E_G$ be the holomorphic reduction of structure group to $N$
constructed in \eqref{en}.

Let $\iota\, :\, \text{ad}(E_N)\, \longrightarrow\, \text{ad}(E_G)$ be the inclusion map.
Take a holomorphic vector bundle $V$ on $M$. Let
$$
(\iota\otimes \text{Id}_V)_*\, :\, H^0(M,\, \text{ad}(E_N)\otimes V)\, \longrightarrow\,
H^0(M,\, \text{ad}(E_G)\otimes V)
$$
be the natural homomorphism. For any $x\,\in\, M$, the adjoint action of $\text{Ad}(E_G)_x$
on $\text{ad}(E_G)_x$ and the trivial action of $\text{Ad}(E_G)_x$ on $V_x$ together produce
a diagonal action of $\text{Ad}(E_G)_x$ on $(\text{ad}(E_G)\otimes V)_x$. A section
$$
\theta\, \in\, H^0(M,\, \text{ad}(E_G)\otimes V)
$$
is said to be fixed by $\mathbb T$ if for every $x\,\in\, M$ and $g\, \in\, {\mathbb T}_x$,
the above action of $g$ on $(\text{ad}(E_G)\otimes V)_x$ fixes $\theta(x)$.

\begin{lemma}\label{lem4}
A section
$$
\theta\, \in\, H^0(M,\, {\rm ad}(E_G)\otimes V)
$$
lies in the image of the above homomorphism $(\iota\otimes {\rm Id}_V)_*$ if and
only if $\theta$ is fixed by $\mathbb T$.
\end{lemma}

\begin{proof}
For any $x\, \in\, M$, the fixed point locus $(\text{ad}(E_G)_x)^{{\mathbb T}_x}$ for the
adjoint action of ${\mathbb T}_x$ on $\text{ad}(E_G)_x$ is $\text{ad}(E_N)_x$. The lemma
follows from this fact.
\end{proof}

\subsection{$G$-Higgs bundles}

In this subsection we set $V\,=\, \Omega^1_M$. The Lie algebra
structure of fibers of $\text{ad}(E_G)$ and the natural projection $\bigotimes^2 \Omega^1_M
\, \longrightarrow\, \Omega^2_M$ together define a homomorphism
$$
(\text{ad}(E_G)\otimes\Omega^1_M)\otimes (\text{ad}(E_G)\otimes\Omega^1_M) \,
\longrightarrow\, \text{ad}(E_G)\otimes\Omega^2_M
$$
which we shall denote by $\bigwedge$.

We recall that a Higgs field on $E_G$ is a holomorphic section
$$
\theta\, \in\, H^0(M,\,\text{ad}(E_G)\otimes\Omega^1_M)
$$
such that $\theta\bigwedge\theta\,=\, 0$ \cite{Si1}, \cite{Si2}, \cite{Hi}. A $G$--Higgs
bundle on $M$ is a pair of the form $(E_G,\, \theta)$, where $E_G$ is a holomorphic
principal $G$--bundle over $M$ and $\theta$ is a Higgs field on $E_G$.

As in Section \ref{sehe}, we assume that $M$ is K\"ahler and it is equipped with a K\"ahler
form $\omega$.

A $G$--Higgs bundle $(E_G,\, \theta)$ is called \textit{stable}
(respectively, \textit{semistable}) if for all triples of the form $(P,\, E_P,\, \chi)$, where
\begin{itemize}
\item{} $P\, \subset\, G$ is proper (not necessarily maximal) parabolic
subgroup,

\item $E_P\, \subset\, E_G$ is a holomorphic reduction of structure group of $E_G$ to
$P$ over a big open subset $U\, \subset\, M$ such that
$$
\theta\vert_U\, \in\, H^0(U,\,\text{ad}(E_P)\otimes\Omega^1_U)\, \subset\,
H^0(U,\,\text{ad}(E_G)\otimes\Omega^1_U)\, ,
$$
and

\item $\chi$ is a strictly anti-dominant character of $P$
which is trivial on the center of $G$,
\end{itemize}
the following holds:
$$
\text{degree}(E_P(\chi))\, > \, 0
$$
(respectively, $\text{degree}(E_P(\chi))\, \geq \, 0$), where $E_P(\chi)$ is the
holomorphic line bundle over $U$ associated to the principal $P$--bundle $E_P$ for the
character $\chi$. (See \cite{Si2}, \cite{BS}.)

A $G$--Higgs bundle $(E_G,\, \theta)$ over $M$ is called \textit{polystable} if either $(E_G,\,theta)$
is stable or there is parabolic subgroup $P\, \subset\, G$ and a holomorphic reduction of
structure group $E_{L(P)}\, \subset\, E_G$ over $M$ to a Levi factor $L(P)$ of $P$, such that
\begin{itemize}
\item $\theta\, in\, H^0(M,\,\text{ad}(E_{L(P)})\otimes\Omega^1_M)$,

\item the $L(P)$--Higgs bundle $(E_{L(P)}, \,\theta)$ is stable, and

\item the reduction of structure group of $E_G$ to $P$ given by the extension of the structure
group of $E_{L(P)}$ to $P$, using the inclusion of $L(P)$ in $P$, is admissible.
\end{itemize}

Fix a maximal compact subgroup $K_G$ as in \eqref{kg}. Let $(E_G,\, \theta)$ be a $G$--Higgs
bundle on $M$. A $C^\infty$ reduction 
of structure group over $M$
$$E_{K_G}\, \subset\, E_G$$
is called a Hermitian--Yang--Mills reduction if the corresponding connection $\nabla$ on $E_{K_G}$
has the property that the curvature ${\mathcal K}(\nabla)$ of $\nabla$ satisfies the equation
\begin{equation}\label{ez2b}
({\mathcal K}(\nabla)+ [\theta,\, \theta^*])\wedge \omega^{d-1} \,=\, c \omega^d\, ,
\end{equation}
where $c$ and $d$ are as in \eqref{ez2}, and $\theta^*$ is the adjoint of $\theta$
(the $\mathbb R$--vector space $\text{Lie}(G)$ has the decomposition
$\text{Lie}(G)\,=\, \text{Lie}(K_G)\bigoplus \mathfrak p$; the adjoint is the conjugate
linear automorphism of the vector space $\text{Lie}(G)$ that acts on $\text{Lie}(K_G)$ as
multiplication by $-1$ and acts on $\mathfrak p$ as the identity map). If $E_{K_G}$ is
a Hermitian--Yang--Mills reduction, then the connection on $E_G$ induced by the corresponding
connection $\nabla$ on $E_{K_G}$ is called a Hermitian--Yang--Mills connection \cite{Si1},
\cite{Si2}, \cite{BS}.

A $G$--Higgs bundle $(E_G,\, \theta)$ admits a Hermitian--Yang--Mills connection if 
and only if $(E_G,\, \theta)$ is polystable, and furthermore, if $(E_G,\, \theta)$ is polystable,
then it has a unique Hermitian--Yang--Mills connection \cite{Si1}, \cite{Hi}, \cite{BS}.

Let $T$, $C$, $N$ and $W$ be as in Theorem \ref{thmse1}.
Let $(E_G,\,\theta)$ be a $G$--Higgs bundle on $M$. Let $\mathbb T$
be a torus subbundles of ${\rm Ad}(E_G)$ such that
\begin{itemize}
\item for some (hence every) $x\, \in\, M$, the
fiber ${\mathbb T}_x$ lies in the conjugacy class of tori in ${\rm Ad}(E_G)$ determined by $T$, and

\item for every $x\, \in\, M$, the action of ${\mathbb T}_x$ on 
$(\text{ad}(E_G)\otimes\Omega^1_M)_x$ fixes the element $\theta(x)$.
\end{itemize}
Let $(E_W,\, \phi,\, E'_C,\, \tau)$ be the corresponding quadruple in Theorem \ref{thmse1}.
Equip $E_W$ with the K\"ahler form $\phi^*\omega$.

The notions of Higgs bundle, polystability and Hermitian--Yang--Mills connection extend to
$E_W$ as before. 

We have the following generalization of Proposition \ref{prop3}:

\begin{proposition}\label{prop3b}
Assume that the $G$--Higgs bundle $(E_G,\, \theta)$ is polystable. Let $\nabla$
be the Hermitian--Yang--Mills connection on $E_G$. Then the following three
hold:
\begin{enumerate}
\item $\theta$ defines a Higgs on the principal $C$--bundle $E'_C$ over $E_W$.

\item The $C$--Higgs bundle $(E'_C,\,\theta)$ is polystable.

\item The Hermitian--Yang--Mills connection $\phi^*\nabla$ on $\phi^*E_G$ preserves
the reduction $E'_C$ of structure group of $\phi^*E_G$ to the subgroup $C$. Furthermore, the
connection on $E'_C$ given by $\phi^*\nabla$ is Hermitian--Yang--Mills connection for
the $C$--Higgs bundle $(E'_C,\,\theta)$.
\end{enumerate}
\end{proposition}

\begin{proof}
In view of Lemma \ref{lem4}, and the uniqueness of the Hermitian--Yang--Mills connection on a
polystable bundle, the proof works along the same line as the proof of Proposition
\ref{prop3}. We omit the details.
\end{proof}

The following is a generalization of Proposition \ref{prop4}:

\begin{proposition}\label{prop4b}
Assume that the $C$--Higgs bundle $(E'_C,\, \theta)$ is polystable. Let $\nabla$
be the Hermitian--Yang--Mills connection on $E'_C$. Assume that the element of $z({\mathfrak k}_c)$
given by the Hermitian--Yang--Mills equation for $(E'_C,\, \theta)$ lies in
$z({\mathfrak k})$. Then the following two
hold:
\begin{enumerate}

\item The $G$--Higgs bundle $(E_G,\,\theta)$ on $M$ is polystable.

\item The Hermitian--Yang--Mills connection on $E_G$ is given by $\nabla$.
\end{enumerate}
\end{proposition}

\begin{proof}
The proof is similar to the proof of Proposition \ref{prop4}.
\end{proof}

\section{Torus for equivariant bundles}\label{seeq}

Let $\Gamma$ be a group and
$$
\varphi_0\, :\, \Gamma\times M\, \longrightarrow\, M
$$
a left--action such that the self--map $x\, \longmapsto\, \varphi_0(\gamma, x)$ of $M$ is
a biholomorphism for every $\gamma\, \in\, \Gamma$. An \textit{equivariant} holomorphic principal
$G$--bundle on $M$ is a pair $(E_G,\, \varphi)$, where $p_0\, :\, E_G\,\longrightarrow\, M$ is a
holomorphic principal $G$--bundle and
$$
\varphi\, :\, \Gamma\times E_G\, \longrightarrow\, E_G
$$
is an action of $\Gamma$ on $E_G$ such that
\begin{itemize}
\item for every $\gamma\, \in\, \Gamma$, the self--map $y\, \longmapsto(\gamma, y)$ of
$E_G$ is a biholomorphism,

\item $p_0\circ\varphi\,=\, \varphi_0\circ (\text{Id}_\Gamma\times p_0)$, and

\item the actions $G$ and $\Gamma$ on $E_G$ commute.
\end{itemize}

Let $(E_G,\, \varphi)$ be an equivariant holomorphic principal $G$--bundle on $M$. For any 
complex manifold $F$ equipped with a holomorphic action of $G$, consider the diagonal action of 
$\Gamma$ on $E_G\times F$ given by the action $\varphi$ on $E_G$ and the trivial action of 
$\Gamma$ on $F$. This diagonal action of $\Gamma$ on $E_G\times F$ produces an action of 
$\Gamma$ on the quotient space of $E_G\times F$ defining the fiber bundle over $M$ associated to 
$E_G$ for $F$. In particular, $\Gamma$ acts on $\text{Ad}(E_G)$; this action of $\Gamma$ on 
$\text{Ad}(E_G)$ preserves the group structure of the fibers of $\text{Ad}(E_G)$.

As in \eqref{e2}, let
\begin{equation}\label{e2b}
{\mathbb T}\, \subset\, \text{Ad}(E_G)\, \stackrel{p}{\longrightarrow}\, M
\end{equation}
be a holomorphic sub-fiber bundle such that
\begin{itemize}
\item the above action of $\Gamma$ on $\text{Ad}(E_G)$ preserves $\mathbb T$,

\item $M\times Z_0(G)\, \subset\, {\mathbb T}$, and

\item for every point $x\, \in\, M$, the
fiber
$$
{\mathbb T}_x\,:=\, {\mathbb T}\cap \text{Ad}(E_G)_x\,\subset\, \text{Ad}(E_G)_x
$$
is a torus (it need not be a maximal torus of $\text{Ad}(E_G)_x$).
\end{itemize}
Fix $T$ as in \eqref{e3}. Then the principal $N$--bundle $E_N\, \subset\, E_G$ 
in \eqref{en} is evidently preserved by the action $\varphi$ of $\Gamma$ on $E_G$.
The resulting action of $\Gamma$ on $E_N$ produces an action of $\Gamma$ on the
quotient $E_W$ of $E_N$ in \eqref{e7}. The projection $\phi$ in \eqref{e7} clearly intertwines
the action of $\Gamma$ on $E_W$ and $M$.

Consider the diagonal action $\Gamma$ on $E_W\times E_G$ constructed using the actions of $\Gamma$ on
$E_W$ and $E_G$. Recall the $\phi^*E_G$ is the submanifold of $E_W\times E_G$ consisting of
all $(x,\, y)\, \in\, E_W\times E_G$ such that $\phi(x)\,=\, p_0(y)$, where $p_0$ and $\phi$ are
the maps in \eqref{p0} and \eqref{e7} respectively. This submanifold of $E_W\times E_G$ is
preserved by the diagonal action of $\Gamma$. Therefore, this action of $\Gamma$ on $E_W\times E_G$
produces an action of $\Gamma$ on $\phi^*E_G$. For this action of $\Gamma$ on $\phi^*E_G$, the
projection $\phi^*E_G\, \longrightarrow\, E_W$ is clearly $\Gamma$--equivariant. Also, the
natural map $\phi^*E_G\, \longrightarrow\, E_G$ is also $\Gamma$--equivariant. The
diagonal action of $\Gamma$ on $E_W\times E_N$ similarly produces an action of $\Gamma$ on
$\phi^*E_N$, where $E_N$ is constructed in \eqref{en}. Note that the inclusion map $\phi^*E_N\,
\hookrightarrow\, \phi^*E_G$ (which is the pullback of the inclusion map in \eqref{en}) is
$\Gamma$--equivariant. Since the principal $N$--bundle $E_N\, \subset\, E_G$
is preserved by the action $\varphi$ of $\Gamma$ on $E_G$, the principal $C$--bundle
$E_C\, \subset\, \phi^*E_N\, \subset\, \phi^*E_G$ constructed in Lemma \ref{lem1} is preserved
by the above action of $\Gamma$ on $\phi^*E_G$.

Conversely, fix a torus $T\, \subset\, G$ a 
containing $Z_0(G)$. The normalizer (respectively, centralizer) of $T$ in $G$ will be denoted by
$N$ (respectively, $C$), while the Weyl group $N/C$ will be denoted by $W$.
Let $$\phi\, :\, E_W\, \longrightarrow\, M$$
be a principal $W$--bundle such that $E_W$ is equipped with an action of $\Gamma$ satisfying the
following conditions:
\begin{itemize}
\item the map $\phi$ is $\Gamma$--equivariant, and

\item the actions of $\Gamma$ and $W$ on $E_W$ commute.
\end{itemize}

Let $(E_G,\, \varphi)$ be an equivariant holomorphic principal $G$--bundle on $M$.
As before, the diagonal action of $\Gamma$ on $E_W\times E_G$ produces an action of $\Gamma$
on $\phi^*E_G$. Let
$$
E'_C\, \subset\, \phi^* E_G
$$
be a holomorphic reduction of structure group of the principal
$G$--bundle $\phi^*E_G$ to $C$ such that the action of $\Gamma$
on $\phi^*E_G$ preserves $E'_C$. Assume that we are further
given a holomorphic action of the complex Lie group $N$ on $E'_C$
$$
\tau\, :\, E'_C\times N\, \longrightarrow\, E'_C
$$
such that
\begin{itemize}
\item the restriction of the map $\tau$ to $E'_C\times C$ is the natural action of 
$C$ on the principal $C$--bundle $E'_C$,

\item the actions of $N$ and $\Gamma$ on $E'_C$ commute, and

\item the diagram of maps
$$
\begin{matrix}
E'_C\times N & \stackrel{\tau}{\longrightarrow} & E'_C\\
\Big\downarrow && ~\Big\downarrow\psi \\
E_W\times W := (E'_C/C)\times (N/C)& \longrightarrow & E_W
\end{matrix}
$$
is commutative, where $E'_C\times N\, \longrightarrow\, (E'_C/C)\times (N/C)\,=\,
(E'_C\times N)/(C\times C)$ is the quotient map, and $\psi$ is the natural
projection of the principal $C$--bundle.
\end{itemize}
Then the action of $\Gamma$ on $\text{Ad}(E_G)$, given by the action $\varphi$ of $\Gamma$ on
$E_G$, preserves the torus sub-bundle ${\mathbb T}\, \subset\, \text{Ad}(E_G)$
constructed in \eqref{e10}.

Therefore, Theorem \ref{thmse1} has the following generalization:

\begin{theorem}\label{thm3}
Let $(E_G,\, \varphi)$ be an equivariant holomorphic principal $G$--bundle on $M$
and $T\, \subset\, G$ a torus
containing $Z_0(G)$. The normalizer (respectively, centralizer) of $T$ in $G$ will be denoted by
$N$ (respectively, $C$), while the Weyl group $N/C$ will be denoted by $W$.
There is a natural bijective correspondence between the following two:
\begin{enumerate}
\item $\Gamma$--invariant torus subbundles $\mathbb T$ of ${\rm Ad}(E_G)$ such that for some (hence
any) $x\, \in\, M$, the fiber ${\mathbb T}_x$ lies in the conjugacy class determined by $T$.

\item Quadruples of the form $(E_W,\, \phi,\, E'_C,\, \tau)$, where $\phi\, :\, E_W\,
\longrightarrow\, M$ is a $\Gamma$--equivariant principal $W$--bundle,
$E'_C\, \subset\, \phi^*E_G$ is a $\Gamma$--invariant holomorphic
reduction of structure group of $\phi^*E_G$ to $C$, and
$$
\tau\,:\, E'_C\times N \, \longrightarrow\, E'_C
$$
is a holomorphic action of $N$ on $E'_C$ extending the natural action of $C$ on the
principal $C$--bundle $E'_C$,
such that the actions of $N$ and $\Gamma$ on $E'_C$ commute, and the diagram of maps
$$
\begin{matrix}
E'_C\times N & \stackrel{\tau}{\longrightarrow} & E'_C\\
\Big\downarrow && ~\Big\downarrow \\
E_W\times W := (E'_C/C)\times (N/C)& \longrightarrow & E_W
\end{matrix}
$$
is commutative.
\end{enumerate}
\end{theorem}

All the results in section \ref{se3} and Section \ref{se4} also extend to the
equivariant set-up.


\end{document}